\documentclass[a4paper]{amsart}

\usepackage{amssymb,amsthm,amsfonts,amsmath,pifont}
\usepackage{enumerate}
\usepackage{color}
\usepackage[Lenny]{fncychap}
\usepackage{pictexwd}
\usepackage{latexsym} 
\usepackage{hyperref}

\newtheorem{theorem}{Theorem}[section]
\newtheorem{lemma}[theorem]{Lemma}
\newtheorem{proposition}[theorem]{Proposition}
\newtheorem{corollary}[theorem]{Corollary}

\title{On the stability of minimal cones in warped products}

\author{K. S. Bezerra}
\address{Departamento de Matem\'atica,
Universidade Federal do Piau\'{\i}, Teresina, Piau\'{\i}, Brazil. 64049-550}
\email{kelton@ufpi.edu.br}

\author{A. Caminha}
\address{Departamento de Matem\'atica, Universidade Federal do Cear\'a, Fortaleza,
Cear\'a, Brazil. 60455-760}
\email{caminha@mat.ufc.br}

\author{B. P. Lima}
\address{Departamento de Matem\'atica,
Universidade Federal do Piau\'{\i}, Teresina, Piau\'{\i}, Brazil. 64049-550}
\email{barnabe@ufpi.edu.br}



\thanks{The second author is partially supported by CNPq. The third author is partially supported by Procad-CNPq.}

\begin{document}

\maketitle

\begin{abstract}
In a seminal paper published in $1968$, J. Simons proved that, for $n\leq 5$, the Euclidean (minimal) cone $CM$, built on a closed, oriented, minimal and non totally geodesic hypersurface
$M^n$ of $\mathbb S^{n+1}$ is unstable. In this paper, we extend Simons' analysis to {\em warped} (minimal) cones built over a closed, oriented, minimal hypersurface 
of a leaf of suitable warped product spaces. Then, we apply our general results to the particular case of the warped product model of the Euclidean sphere, and establish
the unstability of $CM$, whenever $2\leq n\leq 14$ and $M^n$ is a closed, oriented, minimal and non totally geodesic hypersurface of $\mathbb S^{n+1}$.
\end{abstract}

\section{Introduction}

In $1968$, J. Simons (cf. \cite{Simons:1968}) generalized a theorem of F. J. Almgren, Jr. (cf. \cite{Almgren:1966}), showing that, for $n\leq 5$,
the Euclidean cone built over any closed, minimal and non totally geodesic hypersurface of $\mathbb S^{n+1}$ is a minimal unstable hypersurface 
of $\mathbb R^{n+2}$.

If $M^n$ is a hypersurface of $\mathbb S^{n+1}$, the Euclidean cone $CM$ over $M$ is given by the immersion
$\Phi:M^n\times(0,1]\rightarrow\mathbb R^{n+2}$, that sends $(p,t)$ to $tp$. For $0<\epsilon<1$, the $\epsilon$-truncated cone $C_{\epsilon}M$ over $M$ is the restriction of $\Phi$ to $M\times[\epsilon,1]$. 
In \cite{Simons:1968}, it is shown that, if $M^n$ is a closed minimal hypersurface of $\mathbb S^{n+1}$, then $CM\setminus\{0\}$ is 
a minimal hypersurface of $\mathbb R^{n+2}$; moreover, $C_{\epsilon}M$ is compact and such that $\partial (C_{\epsilon}M)=M\cup M_{\epsilon}$, 
where $M_{\epsilon}=\{\epsilon p;\,p\in M\}$.   

In \cite{Caminha:2011}, the second author extended this notion in the following way: let $\overline M_c^{n+2}$ be a Riemannian manifold
whose sectional curvature is constant and equal to $c$. Assume that $\overline M$ admits a closed conformal vector field
$\xi\in\mathfrak X(\overline M)$, with conformal factor $\psi_{\xi}$. If $\xi\neq 0$ on $\overline M$, it is well known that the 
distribution $\{\xi\}^{\perp}$ is integrable, with leaves totally umbilical in $\overline M$. Let $\Xi^{n+1}$ be such a leaf and 
$\varphi:M^n\rightarrow\Xi^{n+1}$ be a closed hypersurface $\Xi$. If $\Psi$ denotes the flow of $\xi/\|\xi\|$, the
compactness of $M$ guarantees the existence of $\epsilon>0$ such that $\Psi$ is well defined in $[-\epsilon,0]\times\varphi(M)$, and the 
mapping $\Phi:M^n\times[-\epsilon,0]\rightarrow\overline M^{n+2}$ that sends $(p,t)$ to $\Psi(t,\varphi(p))$
is also an immersion. By furnishing $M^n\times[-\epsilon,0]$ with the induced metric, we turn $\Phi$ into an isometric immersion such that
$\Phi_{|M^n\times\{0\}}=\varphi$; the Riemannian manifold $M^n\times[-\epsilon,0]$, is the $\epsilon$-truncated cone $C_{\epsilon}M$ over $M$, 
in the direction of $\xi$.

When $\overline M^{n+2}_c$ is a warped product $I\times_fF^{n+1}$, with $I\subset\mathbb R$, it is well known that 
$\xi=(f\circ\pi_I)\partial_t$ ($\pi_I:\overline M\rightarrow I$ being the canonical projection) is a closed conformal vector field on 
$\overline M$, with conformal factor $\psi_{\xi}=f'\circ\pi_I$. If we ask that $f(0)=1$, then $\Xi=\{0\}\times F$, furnished with the 
induced metric, is a leaf of the distribution $\{\xi\}^{\bot}$ and is isometric to $F$. Hence, one can identify an isometric immersion
$\varphi:M^n\rightarrow F^{n+1}$ with the isometric immersion $\widetilde\varphi(p)=(0,\varphi(p))$, from $M^n$ into $\Xi=\{0\}\times F$. 
The flux of $\xi/\|\xi\|$ is given by $\Psi(t,(x,p))=(t+x,p)$, so that $\Phi(p,t)=\Psi(t,\widetilde\varphi(p))=(t,\varphi(p))$ and
$C_{\epsilon}M$ can be identified to the immersion $\Phi:M^n\times[-\epsilon,0]\rightarrow I\times_fF^{n+1}$, that sends $(p,t)$ to
$(t,\varphi(p))$.

Our goal in this paper is to analyse the stability of $C_{\epsilon}M$ when $M^n$ is a closed minimal hypersurface of $F^{n+1}$. 
In doing so, we extend a result of Simons (cf. \cite{Simons:1968}), proving the following assertion (cf. Theorem \ref{instabilidade} and 
equations (\ref{operador l um}) and (\ref{operador l dois})).\\

{\noindent\bf Theorem.}
{\em In the above setting, $C_{\epsilon}M$ is unstable if, and only if, $\lambda_1+\delta_1<0$,
where $\lambda_1$ and $\delta_1$ are, respectively, the first eigenvalues of the linear differential operators 
$\mathcal L_1:C^{\infty}(M)\rightarrow C^{\infty}(M)$ and $\mathcal L_2:C_0^{\infty}[-\epsilon,0]\rightarrow C^{\infty}[-\epsilon,0]$, given by
$$\mathcal L_1(g)=-\Delta g-\|A\|^2g\ \ \text{and}\ \ \mathcal L_2(h)=-f^2h''-nff'h'-c(n+1)f^2h,$$
for $g\in C^{\infty}(M)$ and $h\in C_0^{\infty}[-\epsilon,0]$.}\\

{\noindent Here, as in \cite{Simons:1968}, $C_0^{\infty}[-\epsilon,0]=\{h\in C^{\infty}[-\epsilon,0];\,h(-\epsilon)=h(0)=0\}$.}

Then, we specialize our discussion to the case of {\em spherical cones}. More precisely, we let $\mathbb S^{n+1}$ be the 
equator of $\mathbb S^{n+2}$ with respect to the North pole $N=(0,1)$ of $\mathbb S^{n+2}$, and identify each 
$x\in\mathbb S^{n+1}$ with $(x,0)\in\mathbb S^{n+2}$; if we let $\overline M=\left(-\frac{\pi}{2},\frac{\pi}{2}\right)\times_{\cos t}\mathbb S^{n+1}$, 
then the mapping $(t,x)\mapsto(\cos t)x+(\sin t)N$ defines an isometry between $\overline M$ and $\mathbb S^{n+2}\setminus\{\pm N\}$. 
The $\epsilon-$truncated cone $C_{\epsilon}M$ in $\mathbb S^{n+2}$, built over a closed minimal hypersurface $M^n$ of $\mathbb S^{n+1}$ is 
given by the immersion $\Phi:M^n\times[-\epsilon,0]\rightarrow\mathbb S^{n+2}$ that maps $(x,t)$ to $(\cos t)x+(\sin t)N$. In this setting, we finish
the paper by proving the following result (cf. Theorem \ref{extensao}).\\

{\noindent\bf Theorem.}
{\em Let $M^n$ be a closed, oriented minimal hypersurface of $\mathbb S^{n+1}$. If $2\leq n\leq 14$ and $M^n$ is not totally geodesic, then $CM$ 
is a minimal unstable hypersurface of $\mathbb S^{n+2}$.}

\section{On foliations generated by closed conformal vector fields}\label{preliminares}

In what follows, $\overline M_{c}^{n+k+1}$ is an $(n+k+1)$-dimensional Riemaniann manifold, whose sectional curvature is constant and equal 
to $c$. We assume that $\overline M$ is furnished with a nontrivial closed conformal vector field $\xi$, i.e., 
$\xi\in\mathfrak X(M)\setminus\{0\}$ is such that $\overline\nabla_X\xi=\psi_{\xi}X$, for all $X\in\mathfrak X(M)$, where 
$\psi_{\xi}:\overline M\rightarrow\mathbb R$ is a smooth function, said to be the 
{\em conformal factor} of $\xi$, and $\overline\nabla$ denotes the Levi-Civita connection of $\overline M$. 

From now on, the condition that $\xi\neq 0$ on $\overline M$ will be in force. It is immediate to check (cf. \cite{Caminha:2011}) 
that the distribution $\{\xi^{\perp}\}$ is integrable, with leaves totally umbilical in $\overline M$. Let $\Xi^{n+k}$ be a leaf of such 
distribution, $M^n$ be a closed, $n$-dimensional Riemannian manifold and $\varphi:M^n\rightarrow\Xi^{n+k}$ be an isometric immersion. If 
we let $\Psi(t,\cdot)$ denote the flow of the vector field $\xi/\|\xi\|$, the compactness of $M$ assures that we can choose $\epsilon>0$ 
such that
the map
\begin{eqnarray}\label{equacao cone}
\begin{array}{rccc}
\Phi:&M^n\times[-\epsilon,0]&\longrightarrow&\overline M^{n+k+1}\\
&(p,t)&\longmapsto&\Psi(t,\varphi(p))
\end{array}
\end{eqnarray}
is an immersion. The \textbf{$\epsilon$-truncated cone over $M$, in the direction of $\xi$}, which will be henceforth denoted by 
$C_{\epsilon}M$, is the manifold with boundary $M^n\times[-\epsilon,0]$, furnished with the metric induced by $\Phi$. We observe that 
$C_{\epsilon}M$ is a compact, immersed submanifold of $\overline{M}_{c}^{n+k+1}$, such that $\partial(C_{\epsilon}M)=M\cup M_{\epsilon}$, 
where $M_{\epsilon}=\{\Psi(-\epsilon,\varphi(p));p\in M\}$. At times, if there is no danger of confusion, we shall refer simply to the 
$\epsilon$-truncated cone $C_{\epsilon}M$.

From now on, we will frequently refer to the smooth function $\lambda:M\times[-\epsilon,0]\rightarrow\mathbb{R}$, given by
\begin{equation}\label{eq:funcao lambda}
\lambda(q,t)=\exp\left(\int_{0}^{t}\frac{\psi_{\xi}}{\|\xi\|}(\Psi(s,\varphi(q)))ds\right).
\end{equation}
The following result relates the second fundamental form of $C_{\epsilon}M$ at distinct points along the same generatrix of the cone.

\begin{proposition}\label{segunda forma}
Let $A^{\eta}_q$ denote the shape operator of $\varphi$ at $q$, in the direction of the unit vector $\eta$, normal to $T_qM$ in $T_q\Xi$. Let 
$N$ denote the parallel transport of $\eta$ along the integral curve of $\xi/\|\xi\|$ that passes through $q$. If $A^N_{(q,t)}$ denotes the 
shape operator of $\Phi$ at the point $(q,t)$, in the direction of $N_{(q,t)}$, then
$$\|A^N_{(q,t)}\|=\frac{1}{\lambda(q,t)}\|A^{\eta}_q\|.$$
\end{proposition}

\begin{proof}
Fix a point $p\in M$ and, in a neighborhood $\Omega\subset M$ of $p$, an orthonormal set $\{e_1,\ldots,e_n,\eta\}$ of vector fields, with $e_1,\ldots,e_n$ tangent to $M^n$ and $\eta$ normal to $M^n$ in 
$\Xi^{n+k}$. Further, ask that
$A^{\eta}_p(e_i)=\lambda_ie_i(p)$, for $1\leq i\leq n$. Let $E_1,\ldots,E_n,N$ be the vector fields on $\Phi(\Omega\times(-\epsilon,0])$, respectively obtained from $e_1,\ldots,e_n$ and $\eta$ by 
parallel transport along the integral curves of $\frac{\xi}{\|\xi\|}$ that intersect $\Omega$. 

If we let $\overline R$ denote the curvature operator of $\overline M$ and use the fact that $\overline M$ has constant sectional curvature, such a parallelism gives
\begin{eqnarray}\label{derivada da conexao}
  \nonumber\frac{d}{dt}\langle\overline{\nabla}_{E_{i}}N,E_{k}\rangle &=& \langle\overline{\nabla}_{\frac{\xi}{|\xi|}}\overline{\nabla}_{E_{i}}N,E_{k}\rangle  \\
   \nonumber&=& \frac{1}{\|\xi\|}[\langle
 \overline{R}(\xi,E_{i})N,E_{k}\rangle+\langle\overline{\nabla}_{E_{i}}\overline{\nabla}_{\xi}N,E_{k}\rangle+\langle\overline{\nabla}_{[\xi,E_{i}]}N,E_{k}\rangle]\\
   &=& \frac{1}{\|\xi\|}[\langle
   \overline{R}(\xi,E_{i})N,E_{k}\rangle
   -\langle\overline{\nabla}_{\overline{\nabla}_{E_{i}}\xi}N,E_{k}\rangle] \\
   \nonumber &=&-\frac{\psi_{\xi}}{\|\xi\|}\langle\overline{\nabla}_{E_{i}}N,E_{k}\rangle.
\end{eqnarray}
Moreover, if $D$ denotes the Levi-Civita connection of $\Xi^{n+k}$, then
\begin{equation}\label{conexao no zero}
  \langle\overline{\nabla}_{E_{i}}N,E_{k}\rangle_{(p,0)}=\langle D_{e_{i}}\eta,e_{k}\rangle_{p}=-\langle A^{\eta}(e_{i}),e_{k}\rangle_{p}=-\lambda_{i}\delta_{ik}.
\end{equation}

Equations (\ref{derivada da conexao}) and (\ref{conexao no zero}) compose a Cauchy problem, whose solution is 
$$\langle\overline\nabla_{E_i}N,E_i\rangle_{(p,t)}=-\lambda_i\exp\left(-\int_0^t\frac{\psi_{\xi}}{\|\xi\|}(\varphi(p),s)ds\right)=\frac{-\lambda_i}{\lambda(p,t)}$$
and, for $k\neq i$,
$$\langle\overline\nabla_{E_i}N,E_k\rangle_{(p,t)}=0,$$
for all $t\in(-\epsilon,0]$.

Since $\langle\overline\nabla_{E_i}N,\xi\rangle=-\langle N,\overline\nabla_{E_i}\xi\rangle=-\psi_{\xi}\langle N,E_i\rangle=0$, it follows from 
the previous formulae that, at the point $(p,t)$,
$$A^N(E_i)=-(\overline\nabla_{E_i}N)^{\top}=-\sum_{k=1}^n\langle\overline\nabla_{E_i}N,E_k\rangle E_k-\langle\overline\nabla_{E_i}N,\frac{\xi}{\|\xi\|}\rangle\frac{\xi}{\|\xi\|}=\frac{\lambda_i}{\lambda}E_i,$$
for $1\leq i\leq n$. Finally, taking into account that $A^N(\frac{\xi}{\|\xi\|})=-(\overline{\nabla}_{\frac{\xi}{\|\xi\|}}N)^{\top}=0$, we get
$$\|A^N_{(p,t)}\|^2=\sum_{i=1}^n\left(\frac{\lambda_i}{\lambda(p,t)}\right)^2=\frac{1}{\lambda^2(p,t)}\|A^{\eta}_p\|^2.$$
\end{proof}

\begin{corollary}\label{cone minimo}
The $\epsilon-$truncated cone $C_{\epsilon}M$ is minimal in $\overline M$ if, and only if, $M$ is minimal in $\Xi$.
\end{corollary}

\begin{proof}
If we let $H_{(p,t)}$ be the mean curvature vector of $\Phi$ at $(p,t)$, and $H_p$ be that of $\varphi$ at $p$, it follows from the previous result that 
$\|H_{(p,t)}\|=\frac{1}{\lambda}\|H_p\|$. This proves the corollary.
\end{proof}

The following technical result, which is an adapted version of Theorem $4.1$ of \cite{Caminha:2011}, will be quite useful in the proof of Proposition \ref{laplaciano F}.
In order to state it properly, we let $\nabla$ denote the Levi-Civita connection of $C_{\epsilon}M$. 

\begin{lemma}\label{lema conexao}
Fix $p\in M$ and, in a neighborhood $\Omega$ of $p$ in $M$, an orthonormal frame $(e_1,\ldots,e_n)$, geodesic at $p$. If $E_1,\ldots,E_n$ are 
the vector fields on $\Phi(\Omega\times(-\epsilon,0])$, respectively obtained from $e_1,\ldots,e_n$ by parallel transport along the integral 
curves of $\xi/\|\xi\|$ that intersect $\Omega$, then
\begin{eqnarray}\label{conexao cone}
  \nabla_{E_i}E_i=-\frac{\psi_{\xi}}{\|\xi\|^{2}}\xi
\end{eqnarray}
at $(p,t)$, for all $1\leq i\leq n$.
\end{lemma}

\begin{proof}
Choose vector fields $(\eta_1,\ldots,\eta_k)$ on $\Omega$, such that $(e_1,\ldots,e_n,\eta_1,\ldots,\eta_k)$ is an orthonormal frame adapted 
to the isometric immersion $\varphi$. Also, let $N_1,\ldots,N_k$ be the vector fields on $\Phi(\Omega\times(-\epsilon,0])$, respectively 
obtained from $\eta_1,\ldots,\eta_k$ by parallel transport along the integral curves of $\xi/\|\xi\|$ that intersect $\Omega$. Then, the 
orthonormal frame $(E_1,\ldots,E_n,\frac{\xi}{\|\xi\|},N_1,\ldots,N_k)$ on $\Phi(\Omega\times(-\epsilon,0])$ is adapted to the isometric 
immersion $\Phi$.

We shall compute $\overline\nabla_{E_i}E_i$ at $p$ and take its tangential component along $C_{\epsilon}M$. To this end, note first of all that
\begin{equation}
\langle\overline\nabla_{E_i}E_i,\xi\rangle=-\langle E_i,\overline\nabla_{E_i}\xi\rangle=-\psi_{\xi}.
\end{equation}
As before, letting $\overline R$ denote the curvature operator of $\overline M$, it follows from the parallelism of the $E_i$'s, together with the fact that $\overline M$ has constant sectional curvature, that
\begin{equation}\label{derivada um}
 \begin{split}
\frac{d}{dt}\langle\overline{\nabla}_{E_{i}}E_{i},E_{l}\rangle&\,=\frac{1}{\|\xi\|}\langle\overline{\nabla}_{\xi}\overline{\nabla}_{E_{i}}E_{i},E_{l}\rangle\\
&\,=\frac{1}{\|\xi\|}\langle\overline{R}(\xi,E_{i})E_{i}+\overline{\nabla}_{E_{i}}\overline{\nabla}_{\xi}E_{i}+\overline{\nabla}_{[\xi,E_{i}]}E_{i},E_{l}\rangle\\
&\,=\frac{1}{\|\xi\|}(\langle\overline{R}(\xi,E_{i})E_{i},E_{l}\rangle-\langle\overline{\nabla}_{\overline{\nabla}_{E_{i}}\xi}E_{i},E_{l}\rangle)\\
&\,=-\frac{\psi_{\xi}}{\|\xi\|}\langle\overline{\nabla}_{E_{i}}E_{i},E_{l}\rangle,   
 \end{split}
\end{equation}

Also as before, let $D$ and $\nabla$ respectively denote the Levi-Civita connections of $\Xi^{n+k}$ and $M^n$. Since $(e_1,\ldots,e_n)$ is geodesic at $p$ (on $M$), we get
\begin{equation}\label{no ponto}
\langle\overline\nabla_{E_i}E_i,E_l\rangle_p=\langle D_{e_i}e_i,e_l\rangle_p=\langle\nabla_{e_i}e_i,e_l\rangle_p=0.
\end{equation}
Therefore, by solving Cauchy's problem formed by (\ref{derivada um}) and (\ref{no ponto}), we get
\begin{equation}\label{componente tangente}
\langle\overline{\nabla}_{E_i}E_i,E_l\rangle_{(p,t)}=0,
\end{equation}
for $-\epsilon\leq t\leq 0$.

Analogously to (\ref{derivada um}), we obtain
\begin{equation}\label{derivada dois}
\frac{d}{dt}\langle\overline\nabla_{E_i}E_i,N_{\beta}\rangle=-\frac{\psi_{\xi}}{\|\xi\|}\langle\overline\nabla_{E_i}E_i,N_{\beta}\rangle.
\end{equation}
On the other hand, letting $A_{\beta}:T_pM\rightarrow T_pM$ denote the shape operator of $\varphi$ in the direction of $\eta_{\beta}$ and writing $A_{\beta}e_i=\sum_{j=1}^nh_{ij}^{\beta}e_j$, we get
\begin{equation}\label{no ponto dois}
\langle\overline\nabla_{E_i}E_i,N_{\beta}\rangle_p=\langle D_{e_i}e_i,\eta_{\beta}\rangle_p=\langle A_{\beta}e_i,e_i\rangle=h_{ii}^{\beta}.
\end{equation}
Thus, by solving Cauchy's problem formed by (\ref{derivada dois}) and (\ref{no ponto dois}), we arrive at
\begin{equation}\label{componente normal}
\langle\overline\nabla_{E_i}E_i,N_{\beta}\rangle_{(p,t)}=h_{ii}^{\beta}\exp\left(-\int_0^t\frac{\psi_{\xi}}{\|\xi\|}(s)ds\right)=\frac{h_{ii}^{\beta}}{\lambda(p,t)}.
\end{equation}

Finally, a simple computation shows that
\begin{equation}\label{componente xi}
\langle\overline\nabla_{E_i}E_i,\frac{\xi}{\|\xi\|}\rangle=-\frac{\psi_{\xi}}{\|\xi\|}.
\end{equation}
Therefore, it follows from (\ref{componente tangente}), (\ref{componente normal}) and (\ref{componente xi}) that, at the point $(p,t)$, we have
\begin{equation}
\overline\nabla_{E_i}E_i=-\frac{\psi_{\xi}}{\|\xi\|}\frac{\xi}{\|\xi\|}+\frac{1}{\lambda(p,t)}\sum_{\beta=1}^kh_{ii}^{\beta}N_{\beta}.
\end{equation}

From this equality, (\ref{conexao cone}) follows promptly. 
\end{proof}

Given a smooth function $F\in C^{\infty}(C_{\epsilon}M)$ and $t\in[-\epsilon,0]$, we let $F_t\in C^{\infty}(M)$ be the (smooth) function such that $F_t(p)=F(p,t)$, for all $p\in M$. 
The next result relates the Laplacians of $F$ and $F_t$. 

\begin{proposition}\label{laplaciano F}
In the above notations, for $F\in C^{\infty}(C_{\epsilon}M)$, we have
\begin{equation}\nonumber
 \begin{split}
\Delta F(p,t)&\,=\frac{1}{\lambda^2(p,t)}\left(\Delta F_t(p)-\frac{1}{\lambda(p,t)}\langle {\rm grad}(F_t),{\rm grad}(\lambda_t)\rangle_{p}\right)\\
&\,\,\,\,\,\,\,+n\frac{\lambda'(p,t)}{\lambda(p,t)}\frac{\partial F}{\partial t}+\frac{\partial^2F}{\partial t^2},
 \end{split}
\end{equation}
where $\lambda'$ denotes $\frac{\partial\lambda}{\partial t}$ and ${\rm grad}$ denotes gradient in $M$.
\end{proposition}

\begin{proof}
Fix a point $p\in M$ and, in a neighborhood $\Omega\subset M$ of $p$, an orthonormal frame $(e_1,\ldots,e_n,\eta_1,\ldots,\eta_k)$, adapted 
to $\varphi$, such that $(e_1,\ldots,e_n)$ is geodesic at $p$. As in the proof of Proposition \ref{segunda forma}, parallel transport this 
frame along the integral curves of $\xi/\|\xi\|$ to get vector fields $E_1,\ldots,E_n,N_1,\ldots,N_k$ along 
$\Phi(\Omega\times(-\epsilon,0])$. Then, $(E_1,\ldots,E_n,\frac{\xi}{\|\xi\|},N_1,\ldots,N_k)$ is an orthonormal frame adapted to the immersion $\Phi$.

The Laplacian of $F$ is given by
\begin{equation}\label{laplaciano}
\begin{split}
\Delta F&\,=\sum_{i=1}^nE_i(E_i(F))+\frac{\xi}{\|\xi\|}\left(\frac{\xi}{\|\xi\|}(F)\right)\\
&\,-\sum_{i=1}^n(\nabla_{E_i}E_i)(F)-(\nabla_{\xi/\|\xi\|}\xi/\|\xi\|)(F).
\end{split}
\end{equation}

It follows from Lemma \ref{lema conexao} that
\begin{eqnarray}\label{nablaeiei}
(\nabla_{E_i}E_i)(F)=\left(-\frac{\psi_{\xi}}{\|\xi\|^2}\xi\right)(F)=-\frac{\psi_{\xi}}{\|\xi\|}\frac{\partial F}{\partial t}.
\end{eqnarray}

Now, let us compute the summands $E_i(E_i(F))(q,t)$, where $q\in\Omega$ and $t\in [-\epsilon,0]$. To this end, take a smooth curve 
$\alpha:(-\delta,\delta)\rightarrow M$, such that $\alpha(0)=q$ and $\alpha'(0)=e_{i}(q)$. Then, consider the parametrized surface 
$f:(-\delta,\delta)\times[-\epsilon,0]\rightarrow\overline M$, such that
$$f(s,t)=\Psi(t,\varphi(\alpha(s))),$$
for $(s,t)\in(-\delta,\delta)\times[-\epsilon,0]$. (Note that the image of $f$ is contained in $C_{\epsilon}M$.) Lemma $3.4$ of 
\cite{Manfredo:1990} gives
$$\frac{D}{dt}\frac{\partial f}{\partial s}=\frac{D}{ds}\frac{\partial f}{\partial t}=\frac{D}{ds}\frac{\xi}{\|\xi\|}=\overline{\nabla}_{\frac{\partial f}{\partial s}}\frac{\xi}{\|\xi\|}=\frac{\psi_{\xi}}{\|\xi\|}\frac{\partial f}{\partial s}+\frac{\partial f}{\partial s}\left(\frac{1}{\|\xi\|}\right)\xi,$$
which, in turn, implies
$$\frac{d}{dt}\langle\frac{\partial f}{\partial s},E_j\rangle=\frac{\psi_{\xi}}{\|\xi\|}\langle\frac{\partial f}{\partial s},E_j\rangle.$$

Since $\langle\frac{\partial f}{\partial s},E_j\rangle_{(q,0)}=\langle E_i,E_j\rangle_{(q,0)}=\langle e_i(q),e_j(q)\rangle=\delta_{ij}$, 
in solving the Cauchy problem for $\langle\frac{\partial f}{\partial s},E_j\rangle$ so obtained, we get
$$\langle\frac{\partial f}{\partial s},E_i\rangle_{(q,t)}=\exp\left(\int_0^t\frac{\psi_{\xi}}{\|\xi\|}(q,u)du\right)=\lambda(q,t)$$
and, for $j\neq i$,
$$\langle \frac{\partial f}{\partial s},E_j\rangle_{(q,t)}=0.$$

Moreover, direct computation shows that 
$\frac{d}{dt}\langle\frac{\partial f}{\partial s},\xi\rangle=\frac{\psi_{\xi}}{\|\xi\|}\langle\frac{\partial f}{\partial s},\xi\rangle$; but, 
since $\langle\frac{\partial f}{\partial s},\xi\rangle_{(q,0)}=\langle E_i,\xi\rangle_{(q,0)}=0$, it follows from the uniqueness of the 
solution of a Cauchy problem that $\langle\frac{\partial f}{\partial s},\xi\rangle_{(q,t)}=0$, for $t\in [-\epsilon,0]$.

Since $\frac{\partial f}{\partial s}$ is tangent to the cone, the previous computations show that, at the point $(q,t)$,
\begin{equation}\label{eq:expressao para Ei}
\frac{\partial f}{\partial s}=\sum_{j=1}^n\langle\frac{\partial f}{\partial s},E_j\rangle E_j+\langle\frac{\partial f}{\partial s},\frac{\xi}{\|\xi\|}\rangle\frac{\xi}{\|\xi\|}=\langle\frac{\partial f}{\partial s},E_i\rangle E_i=\lambda E_i.
\end{equation}
Therefore,
$$E_i(F)(q,t)=\frac{1}{\lambda(q,t)}\frac{\partial f}{\partial s}(q,t)(F)=\frac{1}{\lambda(q,t)}dF_t(e_i(q))=\frac{1}{\lambda(q,t)}\langle {\rm grad}(F_t),e_i\rangle_q,$$
for all points $(q,t)\in\Omega\times[-\epsilon,0]$ and all $F\in C^{\infty}(C_{\epsilon}M)$. Thus,
$$E_i(E_i(F))(q,t)=\frac{1}{\lambda(q,t)}\langle {\rm grad}((E_i(F))_t),e_i\rangle_q.$$

On the other hand, at the point $q$ we get
\begin{equation}\nonumber
 \begin{split}
  {\rm grad}((E_i(F))_t)&\,=\frac{1}{\lambda_t}{\rm grad}\langle {\rm grad}(F_t),e_i\rangle+\langle {\rm grad}(F_t),e_i\rangle {\rm grad}\left(\frac{1}{\lambda_t}\right) \\
   &\,=\frac{1}{\lambda_t}{\rm grad}(e_i(F_t))-\langle {\rm grad}(F_t),e_i\rangle\frac{1}{\lambda_t^2}  {\rm grad}(\lambda_t)
 \end{split}
\end{equation}
and, hence,
\begin{equation}\nonumber
 \begin{split}
  E_i(E_i(F))(q,t)&\,=\frac{1}{\lambda(q,t)}\left(\frac{1}{\lambda(q,t)}e_i(e_i(F_t))(q)-\frac{\langle {\rm grad}(F_t),e_i\rangle_q}{\lambda^2(q,t)}e_i(q)(\lambda_t)\right)\\
   &\,=\frac{1}{\lambda^2(q,t)}\left(e_i(e_i(F_t))(q)-\frac{\langle {\rm grad}(F_t),e_i\rangle_q}{\lambda(q,t)}\langle {\rm grad}(\lambda_t),e_i\rangle_q\right),
 \end{split}
\end{equation}
at all points $(q,t)\in\Omega\times[-\epsilon,0]$. 

By using the fact that the frame $(e_1,\ldots,e_n)$ is geodesic at the point $p$ we get, at the point $(p,t)$,
\begin{equation}\label{somaeieif}
\sum_{i=1}^nE_i(E_i(F))=\frac{1}{\lambda^2}\left(\Delta F_t-\frac{1}{\lambda}\langle {\rm grad}(F_t),{\rm grad}(\lambda_t)\rangle\right).
\end{equation}

If we let $(\cdot)^{\top}$ denote orthogonal projection on $T(C_{\epsilon}M)$, we compute
\begin{equation}\nonumber
 \begin{split}
  \nabla_{\xi/\|\xi\|}\xi/\|\xi\|&\,=\frac{1}{\|\xi\|}\nabla_{\xi}\xi/\|\xi\|=\frac{1}{\|\xi\|}\left(\frac{1}{\|\xi\|}\nabla_{\xi}\xi+\xi\left(\frac{1}{\|\xi\|}\right)\xi\right)\\
  &\,=\frac{1}{\|\xi\|}\left(\frac{1}{\|\xi\|}(\widetilde{\nabla}_{\xi}\xi)^{\top}-\frac{\|\xi\|\psi_{\xi}}{\|\xi\|^{2}}\xi\right)\\
  &\,=\frac{1}{\|\xi\|}\left(\frac{1}{\|\xi\|}(\psi_{\xi}\xi)^{\top}-\frac{\psi_{\xi}}{\|\xi\|}\xi\right)=0;
 \end{split}
\end{equation}
hence, $(\nabla_{\xi/\|\xi\|}\xi/\|\xi\|)(F)=0$. 

Substituting this last computation in (\ref{laplaciano}), and taking (\ref{nablaeiei}) and (\ref{somaeieif}) into account, we finally arrive at
\begin{equation}\nonumber
 \begin{split}
  \Delta F(p,t)&\,=\frac{1}{\lambda^2(p,t)}\left(\Delta F_t(p)-\frac{1}{\lambda(p,t)}\langle {\rm grad}(F_t),{\rm grad}(\lambda_t)\rangle_p\right)\\
   &\,\,\,\,\,\,\,+n\frac{\psi_{\xi}}{\|\xi\|}\frac{\partial F}{\partial t}(p,t)+\frac{\partial^2F}{\partial t^2}(p,t),
 \end{split}
\end{equation}
and a simple computation shows that $\frac{\psi_{\xi}}{\|\xi\|}=\frac{\lambda'}{\lambda}$.
\end{proof}

\section{On the unstability of minimal cones}\label{instabilidade secao}

By Corollary \ref{cone minimo}, we know that $M^n$ is minimal in $\Xi^{n+k}$ if, and only if, $C_{\epsilon}M$ is minimal in 
$\overline M_c^{n+k+1}$. Since minimal immersions are precisely the critical points of the area functional with respect to variations that 
fix the boundary, for a given $M$, minimal in $\Xi^{n+k}$, it makes sense to consider the problem of stability of $C_{\epsilon}M$ 
with respect to normal variations that fix its boundary. In this section, we address this problem in the case in which $k=1$, i.e., when 
$M^n$ is a hypersurface of $\Xi^{n+1}$. This will extend the analysis made in \cite{Simons:1968}, where $\overline M=R^{n+2}$, 
$\Xi=S^{n+1}$ and $\xi(x)=x$.

Throughout the rest of this paper, until further notice, we stick to the notations of the previous section. In particular, $\overline M$ 
continues to be of constant sectional curvature, equal to $c$; also, whenever we let $\eta$ denote a unit vector field normal to $M$ in 
$\Xi$, we shall let $N$ denote the unit vector field normal to $C_{\epsilon}M$ in $\overline M$, obtained by parallel transport of $\eta$ 
along the integral curves of $\frac{\xi}{|\xi|}$ that intersect $M$. We start with the following auxiliary result. 

\begin{lemma}\label{elemento de volume}
Let $\Xi^{n+1}$ be oriented by the unit normal vector field $\frac{-\xi}{\|\xi\|}$, and let $M^n$ be a minimal hypersurface of 
$\Xi^{n+1}$, oriented by the unit vector field $\eta\in\mathfrak X(M)^{\bot}\cap\,\mathfrak X(\Xi)$. If $C_{\epsilon}M$ is oriented 
by $N$, then its volume element is given by $\lambda^ndM\wedge dt$, where $dM$ stands for the volume element of $M$.
\end{lemma}

\begin{proof}
Let $(e_1,\ldots,e_n)$ be a positive orthonormal frame, defined in an open set $\Omega\subset M$. If $(\theta_1,\ldots,\theta_n)$ denotes the
corresponding coframe, then $dM=\theta_1\wedge\ldots\wedge\theta_n$ in $\Omega$.

Let $E_1,\ldots,E_n$ be the vector fields on $\Phi(\Omega\times(-\epsilon,\epsilon))$ obtained from the $e_i$'s
by parallel transport along the integral curves of $\xi/\|\xi\|$ that intersect $\Omega$. For $p\in\Omega$, the orthonormal basis
$(e_1,\ldots,e_n,\eta)$ of $T_p\Xi$ is positively oriented; hence, the orthonormal basis $(e_1,\ldots,e_n,\eta,-\frac{\xi}{\|\xi\|})$ 
of $T_p\overline M$ is also positively oriented. It follows that the orthonormal basis $(E_1,\ldots,E_n,N,-\frac{\xi}{\|\xi\|})_{(p,t)}$ 
of $T_{(p,t)}\overline M$ is positively oriented and, thus, $(E_1,\ldots,E_n,\frac{\xi}{\|\xi\|},N)_{(p,t)}$ is also a positively oriented 
orthonormal basis of $T_{(p,t)}\overline M$, for all $(p,t)\in\Omega\times(-\epsilon,\epsilon)$. Therefore, 
$(E_1,\ldots,E_n,\frac{\xi}{\|\xi\|})$ is a positively oriented orthonormal basis of $T_{(p,t)}(C_{\epsilon}M)$.

Now, let $\alpha_i:(-\delta,\delta)\rightarrow M$ be a smooth curve such that $\alpha_i(0)=p$ and $\alpha_i'(0)=e_i(p)$; if
$f_i:(-\delta,\delta)\times(-\epsilon,0]\rightarrow\overline M$ is the parametrized surface such that
$f_i(s,t)=\Psi(t,\varphi(\alpha_i(s)))$, we shaw in (\ref{eq:expressao para Ei}) that
$$E_i(p,t)=\frac{1}{\lambda(p,t)}\frac{\partial f_i}{\partial s}(0,t).$$
By the canonical identification of $T_{(p,t)}(M\times(-\epsilon,0])$ and $T_pM\oplus\mathbb R$, we have 
$$\Phi_*(e_{i}(p)\oplus0)_{(p,t)}=\frac{d}{ds}\big|_{s=0}\Phi(\alpha_i(s),t)=\frac{d}{ds}\big|_{s=0}f_i(s,t)=\frac{\partial f_i}{\partial s}(0,t)$$ 
and, thus,
$$\Phi_*\left(\frac{e_i(p)}{\lambda(p,t)}\oplus 0\right)_{(p,t)}=E_i(p,t).$$
Therefore, by using the canonical identification of $T_{\Phi(p,t)}(C_{\epsilon}M)$ and $\Phi_*(T_{(p,t)}(M\times(-\epsilon,0]))$, we get 
$$\lambda^n(dM\wedge dt)(E_1,\ldots,E_n,\frac{\xi}{\|\xi\|})=\lambda^n(dM\wedge dt)(\frac{e_1}{\lambda}\oplus 0,\ldots,\frac{e_n}{\lambda}\oplus 0,0\oplus\partial t)=1,$$
which concludes the proof.
\end{proof}

Given a minimal isometric immersion $\varphi:M^n\rightarrow\Xi^{n+1}$, the following proposition computes the second variation of area for 
the corresponding $\epsilon-$truncated cone $C_{\epsilon}M$. As usual, for $F\in C^{\infty}(C_{\epsilon}M)$, we let $I(F)$ 
denote the index form of $C_{\epsilon}M$ in the direction of $V=FN$.

\begin{proposition}\label{indice}
Let $M^n$ be a closed, oriented, minimal hypersurface of $\Xi^{n+1}$. Suppose that the function $\lambda(p,t)$ does not depend on the point 
$p$, and let $N(p,t)$ denote the unit normal vector field that orients $C_{\epsilon}M$. If $F\in C^{\infty}(C_{\epsilon}M)$ is such that 
$F(p,-\epsilon)=F(p,0)=0$, for each $p\in M$, then
\begin{equation}\nonumber
 \begin{split}
  I(F)&\,=\int_{M\times [-\epsilon ,0]}F\lambda^{n-2}\Big(-\Delta F_t-n\lambda\lambda'\frac{\partial F}{\partial t}-\lambda^2\frac{\partial^2F}{\partial t^2}\\
        &\hspace{3cm}-c(n+1)\lambda^2F-\|A^{\eta}\|^2F\Big)dM\wedge dt.
 \end{split}
\end{equation}
\end{proposition}

\begin{proof}
It is a classical fact (cf. \cite{Barbosa:88}, \cite{Simons:1968} or \cite{Xin:03}) that
$$I(F)=\int_{C_{\epsilon}M}\left(-F\Delta F-(\overline R+\|A^N\|^2)F^2\right)d(C_{\epsilon}M),$$
where $\overline R=\overline{\text{Ric}}(N,N)$, and $\overline{\text{Ric}}$ denotes the Ricci tensor of $\overline M$.
Therefore, it follows from the formulae of propositions \ref{segunda forma} and \ref{laplaciano F}, together with the fact that $\overline M$
has sectional curvature constant and equal to $c$ and $\lambda(p,t)$ does not depend on $p$, that the integrand of the right hand side equals
\begin{equation}\nonumber
 \begin{split}
   &-F\left(\frac{1}{\lambda^2}\Delta F_t+n\frac{\lambda'}{\lambda}\frac{\partial F}{\partial t}+\frac{\partial^2F}{\partial t^2}\right)-c(n+1)F^2-\|A^{\eta}\|^2\frac{F^2}{\lambda^2}=\\
   &\,=\frac{F}{\lambda^2}\left(-\Delta F_t-n\lambda\lambda'\frac{\partial F}{\partial t}-\lambda^2\frac{\partial^2F}{\partial t^2}-c(n+1)\lambda^2F-\|A^{\eta}\|^2F\right).
 \end{split}
\end{equation}
Finally, it now suffices to apply the result of the previous lemma and integrate on $M\times [-\epsilon,0]$.
\end{proof}

Now, let $C_0^{\infty}[-\epsilon,0]=\{g\in C^{\infty}[-\epsilon,0];\,g(-\epsilon)=g(0)=0\}$. Following \cite{Simons:1968}, the previous 
proposition motivates the introduction of the linear differential operators $\mathcal L_1:C^{\infty}(M)\rightarrow C^{\infty}(M)$ and
$\mathcal L_2:C_0^{\infty}[-\epsilon,0]\rightarrow C^{\infty}[-\epsilon,0]$, given by
$$\mathcal L_1(f)=-\Delta f-\left\|A^{\eta}\right\|^2f\ \ \text{and}\ \ \mathcal L_2(g)=-\lambda^2g''-n\lambda\lambda'g'-c(n+1)\lambda^2g.$$

Standard elliptic theory (cf. \cite{Gilbarg:98}) shows that $\mathcal L_1$ can be diagonalized by a sequence
$(f_i)_{i\geq 1}$ of smooth eigenfunctions, orthogonal in $L^2(M)$ and whose sequence $(\lambda_i)_{i\geq 1}$ of corresponding eigenvalues 
satisfy $\lambda_1\leq\lambda_2\leq\cdots\rightarrow+\infty$; moreover, each $f\in C^{\infty}(M)$ can be uniquely written as
$f=\sum_{i\geq 1}a_if_i$, for some $a_i\in\mathbb R$.

On the other hand, equation $\mathcal L_2(g)=\delta g$, for $\delta\in\mathbb R$, is equivalent to 
$$-\lambda^2g''-n\lambda\lambda 'g'-c(n+1)\lambda^2g-\delta g=0,$$
or (after multiplying both sides by $-\lambda^{n-2}$) yet to
\begin{equation}\label{eq:Sturm-Liouville regular problem under consideration}
(\lambda^ng')'+c(n+1)\lambda^ng+\delta\lambda^{n-2}g=0.
\end{equation}
Hence, the elementary theory of regular Sturm-Liouville problems (cf. \cite{Courant:1989}) shows that $\mathcal L_2$ can also be 
diagonalized by a sequence $(g_i)_{i\geq 1}$ of smooth eigenfunctions, orthogonal in $L^2_w[-\epsilon,0]$ with respect to the weight
$w=\lambda^{n-2}$ and whose sequence $(\delta_i)_{i\geq 1}$ of  corresponding eigenvalues satisfy 
$\delta_1\leq\delta_2\leq\cdots\rightarrow+\infty$; moreover, each $g\in C_0^{\infty}[-\epsilon,0]$ can be uniquely written as
$g=\sum_{i\geq 1}a_ig_i$, for some $a_i\in\mathbb R$.

In view of all of the above, the proof of the following result parallels that of Lemma $6.1.6$ of \cite{Simons:1968}. For the sake of
completeness, we present it here.

\begin{theorem}\label{instabilidade}
With notations as in Proposition \ref{indice}, it is possible to choose $F$ such that $I(F)<0$ if, and only if, $\lambda_1+\delta_1<0$, 
where $\lambda_1$ and $\delta_1$ stand, respectively, to the first eigenvalues of $\mathcal L_1$ and $\mathcal L_2$.
\end{theorem}

\begin{proof}
For a fixed $p\in M$, we have $F(p,\cdot )\in C_0^{\infty}[-\epsilon,0]$. Therefore, the discussion on the diagonalization of $\mathcal L_2$
gives $F(p,t)=\sum_{j\geq 1}a_j(p)g_j(t)$, for some $a_j\in C^{\infty}(M)$; hence, by invoking the discussion on the diagonalization of 
$\mathcal L_1$, we get
$$F(p,t)=\sum_{i,j\geq 1}a_{ij}f_i(p)g_j(t),$$
for some $a_{ij}\in\mathbb R$.

It now follows from the result of Proposition \ref{indice} that
\begin{equation}\nonumber
 \begin{split}
I(F)&\,=\int_{M\times[-\epsilon,0]}\lambda^{n-2}\sum_{i,j\geq 1}a_{ij}f_ig_j\sum_{k,l\geq 1}(a_{kl}\mathcal L_1(f_k)g_l+a_{kl}f_k\mathcal L_2(g_l))dM\wedge dt\\
    &\,=\int_{M\times[-\epsilon,0]}\lambda^{n-2}\sum_{i,j\geq 1}a_{ij}f_ig_j\sum_{k,l\geq 1}a_{kl}(\lambda_k+\delta_l)f_kg_ldM\wedge dt\\
    &\,=\sum_{i,j,k,l\geq 1}a_{ij}a_{kl}(\lambda_k+\delta_l)\int_{M\times[-\epsilon,0]}f_if_kg_jg_l\lambda^{n-2}dM\wedge dt.  
 \end{split}
\end{equation}

From here, the orthogonality conditions on the eigenfunctions of $\mathcal L_1$ and $\mathcal L_2$ easily give
$$I(F)=\sum_{i,j\geq 1}a_{ij}^{2}(\lambda_i+\delta_j)\left(\int_Mf_i^2dM\right)\left(\int_{-\epsilon}^0g_j^2\lambda^{n-2}dt\right).$$
Therefore, if $I(F)<0$, then some factor $\lambda_i+\delta_j$ is negative and, hence, $\lambda_1+\delta_1<0$ (since $\lambda_1\leq\lambda_i$ 
and $\delta_1\leq\delta_j$); conversely, if $\lambda_1+\delta_1<0$, choose $F(p,t)=f_1(p)g_1(t)$ to get $I(F)<0$. 
\end{proof}

For future reference, we recall the standard variational characterization of $\lambda_1$ (cf. \cite{Chavel:1984} or 
\cite{Gilbarg:98}): for a given $f\in C^{\infty}(M)\setminus\{0\}$, let the Rayleigh quotient of $f$ with respect to 
$\mathcal L_1$ be defined by
\begin{equation}\label{eq:Rayleigh quotient for L1}
RQ[f]=\frac{\int_M-f(\Delta f+\left\|A^{\eta}\right\|^2f)dM}{\int_Mf^2dM};
\end{equation}
Then,
\begin{equation}\label{eq:variational characterization of lambda 1}
\lambda_1=\min\{RQ[f];\,f\in C^{\infty}(M)\setminus\{0\}\},
\end{equation}
with equality if, and only if, $f$ is an eigenfunction of $\mathcal L_1$ with respect to $\lambda_1$.

In what concerns $\delta_1$, given $g\in C^{\infty}_0[-\epsilon,0]\setminus\{0\}$, let the Rayleigh quotient of $g$ with respect to 
(\ref{eq:Sturm-Liouville regular problem under consideration}) be defined by
\begin{equation}\label{eq:Rayleigh quotient}
RQ[g]=\frac{\int_{-\epsilon}^0\lambda^n((g')^2-c(n+1)g^2)dt}{\int_{-\epsilon}^0\lambda^{n-2}g^2dt}.
\end{equation}
Then (cf. \cite{Courant:1989}),
\begin{equation}\label{eq:variational characterization of delta 1}
\delta_1=\min\{RQ[g];\,g\in C^{\infty}_0[-\epsilon,0]\setminus\{0\}\},
\end{equation}
with equality if, and only if, $g$ is an eigenfunction of $\mathcal L_2$ with respect to $\delta_1$.

\section{Minimal cones in warped products}\label{aplicacao}

Let $B$ and $F$ be Riemannian manifolds and $f:B\rightarrow\mathbb R$ be a smooth positive function. The {\em warped product} 
$M=B\times_fF$ is the product manifold $B\times F$, furnished with the Riemannian metric
$$g=\pi_B^*(g_B)+(f\circ\pi_B)^2\pi_F^*(g_F),$$
where $\pi_B$ and $\pi_F$ denote the canonical projections from $B\times F$ onto $B$ and $F$ and $g_B$ and $g_F$ denote the Riemannian 
metrics of $B$ and $F$, respectively.

In this section, we shall consider a warped product $\overline M^{n+2}_c=I\times_fF^{n+1}$, with $I\subset\mathbb R$, $f(0)=1$ 
and having constant sectional curvature, equal to $c$. By Proposition $7.42$ of \cite{O'Neill:83}, this last condition amounts to the fact
that $F^{n+1}$ should have constant sectional curvature $k$, such that
$$\frac{f''}{f}=-c=\frac{(f')^{2}-k}{f^{2}}$$
on $I$.

In what concerns our previous discussion of cones, we get the following consequence of Proposition \ref{laplaciano F} when 
$\overline M=I\times_fF$, a warped product for which $I\subset\mathbb R$.

\begin{corollary}\label{cone warped}
Let $\overline M_c^{n+2}=I\times_fF^{n+1}$, with $f(0)=1$. If $M^n$ is a closed Riemannian manifold and $\varphi:M^n\rightarrow F^{n+1}$ 
is an isometric immersion, then
$$\Delta L(t,p)=\frac{1}{f^2(t)}\Delta L_t(p)+n\frac{f'(t)}{f(t)}\frac{\partial L}{\partial t}+\frac{\partial^2L}{\partial t^2},$$
for all $L\in C^{\infty}(I\times_fM^n)$.
\end{corollary}

\begin{proof}
It is a standard fact (cf. \cite{O'Neill:83}) that, in $I\times_fF^{n+1}$, the vector field
$\xi=(f\circ\pi_I)\partial_t$ is closed and conformal, with conformal factor $\psi_{\xi}=f'\circ\pi_I$. Moreover, $\xi\neq 0$, 
since $f$ is positive. The flux $\Psi$ of $\frac{\xi}{\|\xi\|}=\partial_t$ is given by
$$\Psi(t,(t_0,p))=(t+t_0,p),$$
and it is clear that the submanifolds $\{t_0\}\times F^{n+1}$, with $t_0\in I$, are leaves of $\xi^{\perp}$.

Now, let $\varphi:M^n\rightarrow F^{n+1}$ be an isometric immersion from a closed Riemannian manifold $M^n$ into $F^{n+1}$. Since $f(0)=1$,
the leaf $\{0\}\times F^{n+1}$ of $\xi^{\perp}$ (with the metric induced from $I\times_fF^{n+1}$) is isometric to $F^{n+1}$; therefore, we 
can (and do) assume that $\varphi$ takes $M$ into $\{0\}\times F^{n+1}$. The compactness of $M$ guarantees the existence of $\epsilon>0$ 
such that the $\epsilon-$truncated cone $C_{\epsilon}M$ is given by the immersion
$$\Phi(p,t)=\Psi(t,(0,\varphi(p)))=(t,\varphi(p)),$$
for $t\in[-\epsilon,0]$ and $p\in M^n$. (Actually, $\Phi$ continues to be an immersion even if we change $t\in[-\epsilon ,0]$ by $t\in I$.)
Moreover, $C_{\epsilon}M$ is isometric to the warped product $[-\epsilon,0]\times_fM^n$.

In view of the above, the function $\lambda$ of (\ref{eq:funcao lambda}) is such that
$$\lambda(p,s)=\exp\left(\int_0^s\frac{\psi_{\xi}}{\|\xi\|}(\Psi(t,\varphi(p)))dt\right)=\exp\left(\int_0^s\frac{f'}{f}(t)dt\right)=f(s).$$
In particular, $\lambda_s:M^n\rightarrow\mathbb R$ is constant, for all $s\in[-\epsilon,0]$, and it suffices to apply the result of 
Proposition \ref{laplaciano F}.
\end{proof}

From now on, let $M^n$ be a closed, minimal and non totally geodesic hypersurface of $F^{n+1}\approx\{0\}\times F^{n+1}$. According to the 
proof of the previous corollary, we shall identify the $\epsilon-$truncated cone $C_{\epsilon}M$ with the warped product 
$[-\epsilon,0]\times_fM^n$, canonically immersed into $\overline M^{n+2}_c$.

If (as before) $N(t,p)$ stands for the unit normal vector field of $C_{\epsilon}M$ and $G\in C^{\infty}(C_{\epsilon}M)$ is such that 
$G(-\epsilon,p)=G(0,p)=0$ for each $p\in M$, then Proposition \ref{indice} gives
\begin{equation}\label{indice warped}
 \begin{split}
  I(G)&\,=\int_{M\times[-\epsilon,0]}Gf^{n-2}\Big(-\Delta G_t-nff'\frac{\partial G}{\partial t}-f^2\frac{\partial^2G}{\partial t^2}\\
 &\,\hspace{3cm}-c(n+1)f^2G-\|A\|^2G\Big)dM\wedge dt,
 \end{split}
\end{equation}
where $\|A\|$ stands for the norm of the second fundamental form of the immersion of $M^n$ into $F^{n+1}$ and $\Delta$ for the Laplacian
operator of $M^n$. 

In this case, the linear differential operators $\mathcal L_1:C^{\infty}(M)\rightarrow C^{\infty}(M)$ and 
$\mathcal L_2:C_0^{\infty}[-\epsilon,0]\rightarrow C^{\infty}[-\epsilon,0]$ are given by
\begin{equation}\label{operador l um}
\mathcal L_1(g)=-\Delta g-\|A\|^2g,
\end{equation}
for $g\in C^{\infty}(M)$, and
\begin{equation}\label{operador l dois}
\mathcal L_2(h)=-f^2h''-nff'h'-c(n+1)f^2h,
\end{equation} 
for $h\in C_0^{\infty}[-\epsilon,0]$. 

We want to apply Theorem \ref{instabilidade} to the case in which $\overline M_c^{n+2}$ is the Euclidean sphere $\mathbb S^{n+2}$. 
To this end, let $I=\left(-\frac{\pi}{2},\frac{\pi}{2}\right)$, $f(t)=\cos t$, $F^{n+1}=\mathbb S^{n+1}$,
$N=(0,\ldots,0,1)\in\mathbb S^{n+2}$ and consider $\mathbb S^{n+1}$ as the equator of $\mathbb S^{n+2}$ which has $N$ as North pole; also, 
identify $x=(x_1,\ldots,x_{n+2})\in\mathbb S^{n+1}$ to the point $x=(x_1,\ldots,x_{n+2},0)\in\mathbb S^{n+2}$. With these conventions, the 
map 
$$(t,x)\mapsto(\cos t)x+(\sin t)N$$ 
defines an isometry between $\left(-\frac{\pi}{2},\frac{\pi}{2}\right)\times_{\cos t}\mathbb S^{n+1}$ and $\mathbb S^{n+2}\setminus\{\pm N\}$. 

Once again, let $\varphi:M^n\rightarrow\mathbb S^{n+1}$ be a closed, minimal, non totally geodesic hypersurface of $\mathbb S^{n+1}$. The 
$\epsilon-$truncated cone $C_{\epsilon}M$ can be seen as the image of the isometric immersion
\begin{equation}\label{cone esfera}
\begin{array}{rrcl}
\Phi:&[-\epsilon,0]\times M^n&\longrightarrow&\mathbb S^{n+2}\\
     &(t,x)&\longmapsto&(\cos t)x+(\sin t)N.
\end{array}
\end{equation}

In order to get an upper estimate for $\lambda_1$, recall from (\ref{eq:Rayleigh quotient for L1}) and 
(\ref{eq:variational characterization of lambda 1}) that 
\begin{equation}\label{eq:estimating lambda 1 again}
\lambda_1\leq\frac{\int_M-g(\Delta g+\|A\|^2g)dM}{\int_Mg^2dM},
\end{equation}
for any $g\in C^{\infty}(M)\setminus\{0\}$. Following \cite{Simons:1968}, let $\tau>0$ and $g_{\tau}=(\|A\|^2+\tau)^{1/2}$. Simons' formula for 
$\Delta(\|A\|^2)$ (cf. \cite{Caminha:2006} or \cite{Simons:1968} -- recall that $F$ is also of constant sectional curvature) easily gives
$$g_{\tau}\Delta g_{\tau}\geq n\|A\|^2-\|A\|^4.$$
Hence, by taking $g_{\tau}$ in place of $g$ in (\ref{eq:estimating lambda 1 again}), we arrive at
$$\lambda_1\leq-\frac{\int_M(n+\tau)\|A\|^2dM}{\int_M(\|A\|^2+\tau)dM},$$
By letting $\tau\rightarrow 0$, and taking into account that $\int_M\|A\|^2dM>0$ (since $M$ is not totally geodesic), we get $\lambda_1\leq-n$.

In what concerns $\delta_1$, equation (\ref{operador l dois}) gives
\begin{eqnarray*}
\mathcal L_2(h)=-(\cos^2t)h''+n(\sin t\cos t)h'-(n+1)(\cos^2t)h,
\end{eqnarray*} 
so that (arguing as in the discussion that precedes the statement of Theorem \ref{instabilidade}) $\mathcal L_2(h)=\delta h$ is equivalent to
\begin{equation}\label{sturm problem}
((\cos^nt)h')'+(n+1)(\cos^nt)h+\delta(\cos^{n-2}t)h=0.
\end{equation} 
It now follows from (\ref{eq:Rayleigh quotient}) and (\ref{eq:variational characterization of delta 1}) that
$$\delta_1\leq\frac{\int_{-\epsilon}^0(\cos^nt)((h')^2-(n+1)h^2)dt}{\int_{-\epsilon}^0(\cos^{n-2}t)h^2dt},$$
for every $h\in C_0^{\infty}[-\epsilon,0]\setminus\{0\}$. 

By taking 
$$h(t)=\frac{\sin(\frac{\pi}{\epsilon}t)}{\sqrt{\cos^{n-2}t}}$$
(which satisfies the boundary conditions), direct computations show that $h(t)^2\cos^{n-2}t=\sin^2\left(\frac{\pi}{\epsilon}t\right)$,
\begin{equation}\nonumber
\begin{split}
(\cos^nt)h'(t)^2&\,=\frac{\pi^2}{\epsilon^2}\cos^2\left(\frac{\pi}{\epsilon}t\right)\cos^2t+\frac{(n-2)^2}{4}\sin^2\left(\frac{\pi}{\epsilon}t\right)\sin^2t\\
                  &\,\,\,\,\,\,\,+\frac{n-2}{4}\sin\left(\frac{2\pi}{\epsilon}t\right)\sin(2t),
\end{split}
\end{equation}
and
$$(n+1)(\cos^nt)h(t)^2=(n+1)(\cos^2t)\sin^2\left(\frac{\pi}{\epsilon}t\right).$$
Therefore,
$$\delta_1\leq\frac{I_1-I_2}{I_3},$$
where 
\begin{equation}\nonumber
\begin{split}
I_1&\,=\frac{\pi^2}{\epsilon^2}\int_{-\epsilon}^0\cos^2\left(\frac{\pi}{\epsilon}t\right)\cos^2tdt+\frac{(n-2)^2}{4}\int_{-\epsilon}^0\sin^2\left(\frac{\pi}{\epsilon}t\right)\sin^2tdt\\
   &\,\,\,\,\,\,\,+\frac{n-2}{4}\int_{-\epsilon}^0\sin\left(\frac{2\pi}{\epsilon}t\right)\sin(2t)dt,
\end{split}
\end{equation}
$$I_2=(n+1)\int_{-\epsilon}^0(\cos^2t)\sin^2\left(\frac{\pi}{\epsilon}t\right)dt\ \ \text{and}\ \ I_3=\int_{-\epsilon}^0\sin^2\left(\frac{\pi}{\epsilon}t\right)dt.$$

Finally, we observe that $\lim_{\epsilon\rightarrow\frac{\pi}{2}}I_1=\frac{\pi}{2}\left(1+\left(\frac{n-2}{4}\right)^2\right)$,
$\lim_{\epsilon\rightarrow\frac{\pi}{2}}I_2=(n+1)\frac{\pi}{8}$ and $\lim_{\epsilon\rightarrow\frac{\pi}{2}}I_3=\frac{\pi}{4}$, so that, for $\epsilon>0$ sufficiently close to $\frac{\pi}{2}$, we have
$$\lambda_1+\delta_1\leq\frac{n^2}{8}-2n+2.$$
Since this quadratic polynomial is negative for $2\leq n\leq 14$, we have proved the following result.

%By taking $h(t)=\sin\left(\frac{\pi}{\epsilon}t\right)$, we get
%\begin{eqnarray*}
%\delta_1\leq\frac{\int_{-\epsilon}^0(\cos^nt)(\frac{\pi^{2}}{\epsilon^{2}}\cos^2(\frac{\pi}{\epsilon}t)-(n+1)\sin^2(\frac{\pi}{\epsilon}t))dt}{\int_{-\epsilon}^{0}(\cos^{n-2}t)\sin^2(\frac{\pi}{\epsilon}t)dt}:=\delta(\epsilon).
%\end{eqnarray*} 

%For $n=2$, direct computation shows that $\lim_{\epsilon\rightarrow\frac{\pi}{2}}\delta(\epsilon)=\frac{1}{2}$. Therefore, by choosing 
%$\epsilon$ sufficiently close to $\frac{\pi}{2}$, we get 
%$$\lambda_1+\delta_1<-2+\frac{1}{2}<0,$$
%thus proving the following result.

\begin{theorem}\label{extensao}
Let $M^n$ be a closed, oriented minimal hypersurface of $\mathbb S^{n+1}$. If $2\leq n\leq 14$ and $M^n$ is not totally geodesic, then $CM$ is a minimal unstable hypersurface of $\mathbb S^{n+2}$. 
\end{theorem}

\end{document}